\documentclass[12pt]{article}
\usepackage{amsmath}
\usepackage{amscd}
\usepackage{amsthm}
\usepackage{amssymb, fullpage} 
\usepackage{eufrak}
\usepackage{authblk}
\usepackage{latexsym}
\usepackage{euscript}
\usepackage{hyperref}
\usepackage{epsfig}
\usepackage{tikz}
\usepackage{graphics}
\usepackage{array}
\usepackage{enumerate}
\theoremstyle{theorem}
\newtheorem{theorem}{Theorem}[section]
\theoremstyle{corollary}
\newtheorem{corollary}{Corollary}[section]
\theoremstyle{lemma}
\newtheorem{lemma}{Lemma}[section]
\theoremstyle{definition}
\newtheorem{definition}{Definition}[section]
\theoremstyle{conjecture}

\theoremstyle{proof}

\theoremstyle{remark}
\newtheorem*{rem}{Remark}
  \theoremstyle{example}
\newtheorem*{example}{Example}

\DeclareMathAlphabet{\mathpzc}{OT1}{pzc}{m}{it}
\begin{document}
\title{On the normalized spectrum of threshold graphs}
\author[1,2]{\rm Anirban Banerjee}
\author[1]{\rm Ranjit Mehatari}
\affil[1]{Department of Mathematics and Statistics}
\affil[2]{Department of Biological Sciences}
\affil[ ]{Indian Institute of Science Education and Research Kolkata}
\affil[ ]{Mohanpur-741246, India}
\affil[ ]{\textit {\{anirban.banerjee, ranjit1224\}@iiserkol.ac.in}}
\maketitle
\begin{abstract}
In this article we investigate normalized adjacency eigenvalues {(simply normalized eigenvalues)} and normalized adjacency energy of connected threshold graphs. A threshold graph can always be represented as a unique binary string. Certain eigenvalues are obtained directly from its binary representation and the rest of the eigenvalues are evaluated from its normalized equitable partition matrix. Finally, we characterize threshold graphs with at most five distinct eigenvalues.
\end{abstract}
\section{Introduction}
In this paper, we only consider simple, connected, undirected finite graphs. Let $\Gamma=(V,E)$ be an $n$ vertex graph with $V=\{1,2,\ldots,n\}$.  Two vertices $i,j\in V$ are called neighbors, $i\sim j$, when they are connected by an edge  in $E$. Let
$d_i$ denote the degree of a vertex $i$. Let $A$ be the adjacency matrix \cite{Cve} of $\Gamma$ and let $D$ be the diagonal matrix of vertex degrees of $\Gamma$. The \textit{normalized adjacency} matrix $\mathcal{A}$ of $\Gamma$ is defined by $\mathcal{A}=D^{-1}A$ which is  similar to the matrix $R=D^{-\frac{1}{2}}AD^{-\frac{1}{2}}$, called the \textit{Randi\'c matrix} \cite{Boz2} of $\Gamma$. Thus the matrices $R$ and $\mathcal{A}$ have same eigenvalues. The matrix $\mathcal{A}$ is a row-stochastic matrix, often called the \textit{transition matrix} of $\Gamma$. For any function $f:V(\Gamma)\longrightarrow \mathbb{R}$, $\mathcal{A}$ is given by 
$$\mathcal{A} f(i)=\frac{1}{d_i}\sum_{j\sim i}f(j),\text{ for all }i\in V(\Gamma).$$ 
Furthermore, $\mathcal{A}$  is self-adjoint with respect to the inner product  defined by $$\langle u,v\rangle=\sum_id_iu(i)v(i).$$

{The} normalized adjacency matrix has a direct connection with the \textit{normalized Laplacian matrix} $\mathcal{L}=I_n-R$ studied in \cite{Chung}, and with $\Delta=I_n-\mathcal{A}$ studied in \cite{Ban1,Ban3}. Thus, for any graph $\Gamma$, if $\lambda$ is an eigenvalue of the normalized Laplacian matrix then $1-\lambda$ is an eigenvalue of the normalized adjacency matrix. The matrix $\mathcal{A}$ has some nice properties, such as, $1$ is always an eigenvalue of $\Gamma$ with $e=\left[\begin{array}{cccc}
1&1&\ldots&1
\end{array}\right]^T$ as its corresponding eigenvector. The eigenvalues of $\mathcal{A}$ are bounded below by $-1$, and the lower bound is attained if and only if the graph is bipartite. So, $-1$ is the trivial lower bound for eigenvalues of $\mathcal{A}$. For the results on the non-trivial bounds of eigenvalues we refer to \cite{Ban2,Chung,Rojo}. If $\lambda$ is an eigenvalue of $\mathcal{A}$ then there is nonzero function $f$ satisfying $\mathcal{A} f-\lambda f=0$, which yields the eigenvalue equation
$$\frac{1}{d_i}\sum_{j\sim i}f(j)=\lambda d_if(i),\textit{ }\forall i\in\Gamma.$$
Since $\mathcal{A}$ is self-adjoint, thus any eigenvector $f$ other than $e$ satisfies 
$$\sum_{i=1}^nd_if(i)=0.$$
\begin{definition}
A graph is called \textit{threshold graph} if it does not contain $C_4$, $P_4$ or $2K_2$ as its induced subgraph. 
\end{definition}

A threshold graph can be obtained from a single vertex by repeatedly  performing  one of the  two graph operations, namely, (a) addition of a single isolated vertex to the graph or (b) addition of a single dominating vertex to the graph, i.e., a single vertex which connects to all other existing vertices. Thus, an $n$-vertex threshold graph is represented by a binary string $b_1\ldots b_n$ where $b_1=0$ and, for $2\leq i\leq n$, $b_i=0$ if the vertex $i$ is added as an isolated vertex, and $b_i=1$ if the same is added as a dominating vertex.

Hence, for a connected threshold graph with $n\geq 2$ vertices $b_n=1$. The binary representation of a threshold graph is unique, and conversely, for any above mentioned binary string there is exactly one threshold graph. Thus, there are exactly $2^{n-2}$ distinct connected threshold graphs of order $n$. 

{Threshold graphs were introduced in 1977 \cite{Chv,Hen} and they became popular as they have numerous applications in computer science and psychology \cite{Mah}. Recently many researchers investigated} the eigenvalues  of different matrix representations of threshold graphs. In \cite{Sci}, Sciriha and Farrugia studied the spectral properties of adjacency matrix of a threshold graph. Bapat \cite{Ba} found the determinant of the adjacency matrix of a threshold graph. He showed that the nullity of a threshold graph can be calculated directly from its binary string. Jacobs, Trevisan and Tura published several articles on the adjacency spectrum of threshold graphs \cite{Jac1,Jac2,Jac3}. They developed algorithms to locate the eigenvalues \cite{Jac1} and to compute the characteristic polynomial \cite{Jac2} of a threshold graph. They showed that the adjacency matrix of a threshold graph does not have any eigenvalue in $(-1,0)$. 

{From now on, a threshold graph will be considered as connected.} Without loss of generality, we denote the binary string of a  threshold graph $\Gamma$ by $b=0^{s_1}1^{t_1}\ldots0^{s_k}1^{t_k}$, where $s_i,t_i\in \mathbb{N}$. Let $s=\sum s_i$ and $t=\sum t_i$, respectively, denote the number of $0$'s and the number of $1$'s in the string. Also let $S_i=\sum_{j=1}^i s_j$ and $T_i=\sum_{j=1}^i t_j$.\\
 
 For a square matrix $M$ the triple $(n_-(M),n_0(M),n_+(M))$ is called the inertia $M$, where  $n_-(M)$ and $n_+(M)$ denote the number of negative and positive eigenvalues, respectively, whereas, $n_0(M)$ is the nullity of $M$. Bapat \cite{Ba} and Jacobs et al.~\cite{Jac3} described the inertia of the adjacency matrix for threshold graphs. Now we give two simple results related to inertia and determinant of $\mathcal{A}$ of threshold graphs. 

\begin{figure}[h]
\centering
\includegraphics[width=\textwidth]{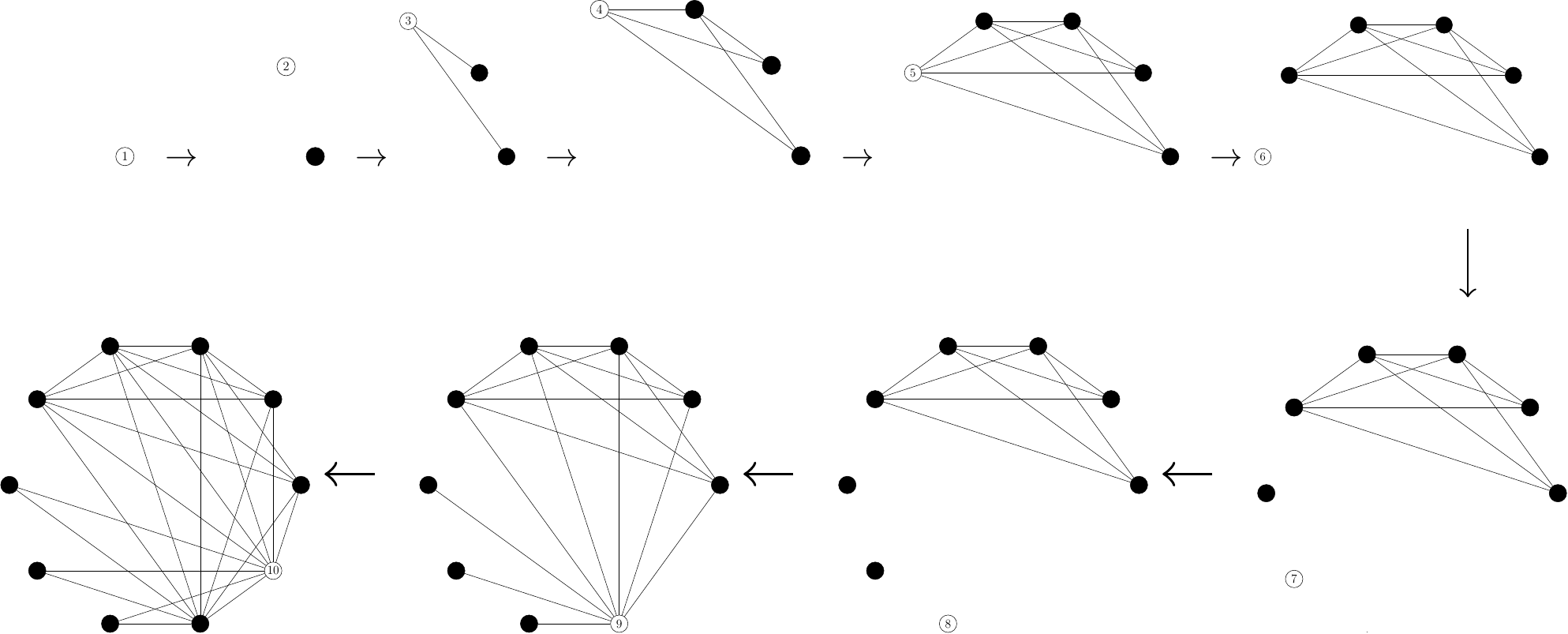}
\caption{Creation of the threshold graph from the string $0011100011.$} \label{tldfig1}
\end{figure}

\begin{theorem}
\label{tldthm1}
Let $0^{s_1}1^{t_1}\ldots0^{s_k}1^{t_k}$ be the binary string of a threshold graph $\Gamma$. Then
\begin{itemize}
\item[1.]
$n_-(\mathcal{A})=t,$
\item[2.]
$n_0(\mathcal{A})=s-k,$
\item[3.]
$n_+(\mathcal{A})=k$.
\end{itemize}
\end{theorem}
\begin{proof}
Let $A$ denote the adjacency matrix for $\Gamma$. We have $A=D^{\frac{1}{2}}RD^{\frac{1}{2}}$. Thus, our required result can be obtained from Theorem 1 and Theorem 2 of \cite{Jac3} and \textit{Sylvester's law of inertia}.
\end{proof}

\begin{theorem}
\label{tldthm2}
Let $0^{s_1}1^{t_1}\ldots0^{s_k}1^{t_k}$ be the binary string of a threshold graph $\Gamma$. Then
$$det(\mathcal{A})=\begin{cases}
\frac{(-1)^tt_1t_2\ldots t_k}{t(t-T_1)\ldots t_k(t+S_1-1)^{t_1}(t+S_2-1)^{t_2}\ldots (n-1)^{t_k}}, &\text{ if }s_i=1\ \forall i\\
0,&\text{ otherwise.}
\end{cases}$$
\end{theorem}
\begin{proof}
{
We have $det(\mathcal{A})=det(D)^{-1}det(A)$. By Theorem 5 of \cite{Ba}, $$det(A)=\begin{cases}(-1)^tt_1t_2\ldots t_k,&\text{if }s_i=1\ \forall i\\
0, &\text{otherwise}.
\end{cases}$$ Thus the result follows, since the graph $\Gamma$ has $s_1$ vertices of degree $t$, $t_1$ vertices of degree $t+S_1-1$,  $s_2$ vertices of degree $t-t_1$, $t_2$ vertices of degree $t+S_2-1$, and so on.}
\end{proof}
\begin{corollary}
Let $\Gamma$ be the threshold graph with binary string $01\ldots01$, then
$det(\mathcal{A})=(-1)^\frac{n}{2}\frac{2}{n!}$.
\end{corollary}

\section{Eigenvalues of threshold graphs}
Let $0^{s_1}1^{t_1}\ldots0^{s_k}1^{t_k}$ be the binary string of a threshold graph $\Gamma$. Let $\pi=\lbrace V_{s_1},V_{t_1},\ldots,V_{t_k}\rbrace$ be a partition of the vertex set of $\Gamma$, such that, $V_{s_1}$ contains first $s_1$ vertices of $\Gamma$, $V_{t_1}$ contains next $t_1$ vertices of $\Gamma$ and so on. The partition $\pi$ is an equitable partition of $\Gamma$. If there is some $t_i>1$, then certain eigenvalues of $\mathcal{A}$ can be estimated directly from the string $b$.
\begin{lemma}
\label{tldlem1}
Let $0^{s_1}1^{t_1}\ldots0^{s_k}1^{t_k}$ be the binary string of a threshold graph $\Gamma$. Then, for each $t_i>1$, $-\frac{1}{t+S_i-1}$ is an eigenvalue of $\mathcal{A}$ of multiplicity at least $t_i-1$.
\end{lemma}
\begin{proof}
We show that, for $t_i>1$, there exist $t_i-1$ number of mutually orthogonal set of eigenvectors.\\
Let $X_{p}^{(l)}$ denote the $p$-tuples such that
$$X_p^{(l)}(j)=\begin{cases}
1,&\text{if }j<k,\\
-l+1,&\text{if }j=l,\\
0,&\text{otherwise.}
\end{cases}$$
Now, for $t_i>1$, we construct the functions $f_{t_i}^{(2)},\ldots,f_{t_i}^{(t_i)},$ where $$f_{t_i}^{(j)}=\left[\begin{array}{ccccccccc}
O_{s_1}&O_{t_1}&\cdots&O_{s_i}&X_{t_i}^{(j)}&O_{s_{i+1}}&\cdots&O_{t_k}
\end{array}\right]^T.$$
For $2<j_1<j_2\leq t_i,$
\begin{eqnarray*}
\langle f_{t_i}^{(j_1)},f_{t_i}^{(j_2)}\rangle&=&\sum _{x\in V(\Gamma)} d_xf_{t_i}^{(j_1)}(x)f_{t_i}^{(j_1)}(x) \\&&=(t+S_i-1)(1+\cdots+1-j_1+1)\\
&&=0.
\end{eqnarray*}
Now, we show that, for $2\leq j\leq t_i$, $f_{t_i}^{(j)}$ is an eigenvector of $\mathcal{A}$. We have, 
$$\mathcal{A}f_{t_i}^{(j)}(x)=\begin{cases}
0,\ &\text{if }x\notin V_{t_i},\\
-\frac{1}{t+S_i-1}f_{t_i}^{(j)}(x),\ &\text{if}\ x\in V_{t_i}.
\end{cases}$$
Thus, $f_{t_i}^{(2)},\ldots,f_{t_i}^{(t_k)}$ are the eigenvectors corresponding to the eigenvalue $-\frac{1}{t+S_i-1}$. Hence the multiplicity of the eigenvalue $-\frac{1}{t+S_i-1}$ is at least $t_i-1.$
\end{proof}
\begin{rem}
If $s>k$, we can construct the eigenvectors corresponding to the eigenvalue $0$. For $s_i>1$, we construct the functions $f_{s_i}^{(2)},\ldots,f_{s_i}^{(s_i)}$ as
$${f_{s_i}^{(j)}}=\left[\begin{array}{ccccccccc}
O_{s_1}&O_{t_1}&\cdots&O_{t_{i-1}}&X_{s_i}^{(j)}&O_{t_i}&\cdots&O_{t_k}
\end{array}\right]^T,$$
which provide an orthogonal set of $s-k$ eigenvectors corresponding to the eigenvalue 0 {of the matrix $\mathcal{A}$}. Let $T$ be the collection of {the normalized eigenvalues} of a threshold graph (with the string $0^{s_1}1^{t_1}\ldots0^{s_k}1^{t_k}$) which are obtained directly from the string. Then \begin{eqnarray*}
T=\begin{cases}
\{0\}\cup\{\frac{-1}{t+s_i-1}:t_i>1\},&\text{if }s>k,\\
\{\frac{-1}{t+s_i-1}:t_i>1\},&\text{if }s=k.
\end{cases}
\end{eqnarray*}
So, $T=\emptyset$ if $s_i,t_i=1,\ \text{for } 1\leq i\leq k.$
\end{rem}
\begin{figure}[h]
\centering
\includegraphics[width=12cm]{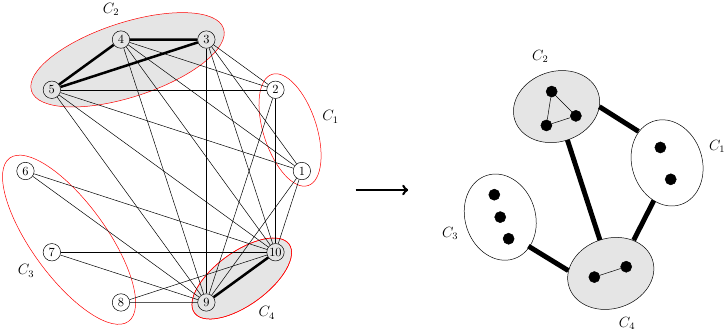}
\caption{Partition representation for the binary string $0011100011.$} \label{tldfig3}
\end{figure}
Let $\Gamma$ be a graph and $\pi=\lbrace C_1,C_2,\ldots,C_m\rbrace$ be an equitable partition of $\Gamma$. Let $B_\pi$ be the partition matrix for the partition $\pi$. The matrix $B_\pi=(b_{ij})$ is an $m\times m$ matrix whose $(i,j)$-th entry equals to the number of connections from a vertex $v\in C_i$ to the vertices of $C_j$.  We define $\mathcal{B}_\pi=D_\pi^{-1}B_\pi$ where $D_\pi$ is the diagonal matrix with $i$th {diagonal entry $d(C_i)=\sum_jb_{ij}$} which is equal to the degree of each vertex in $C_i$. Let $P$ be the characteristics matrix  of $\pi$. Hence $P$ is an $n\times m$ matrix and $P^TP$ is the diagonal matrix whose $i$th diagonal {entry} is $|C_i|$. {The adjacency eigenvalues of equitable partitions have been discussed previously in literature \cite{Cve,Sci}. We use similar concept to estimate the normalized adjacency eigenvalues of an
equitable partition.} Now, we introduce the following lemma.
\begin{lemma}
\label{tldlem2}
Let $\Gamma$ be a graph and let $\pi$ be an equitable partition of $\Gamma$. Then $\mathcal{A} P=P\mathcal{B}_\pi$.
\end{lemma}
\begin{proof}
We prove that both the matrices $\mathcal{A} P$ and $P\mathcal{B}_\pi$ have the same entries. The $(k,j)$th entry of $\mathcal{A} P$ is 
$$(\mathcal{A} P)_{kj}= \frac{|\lbrace v:v\sim k,v\in C_j\rbrace|}{d_k}.$$
If $k\in C_i$ then $(\mathcal{A} P)_{kj}=(\mathcal{B}_\pi)_{ij}$. Since $k\in C_i$, the only nonzero entry of $k$th row of $P$ is in $i$th column. Therefore $(P\mathcal{B}_\pi)_{kj}=(\mathcal{B}_\pi
)_{ij}$.
Hence, $\mathcal{A} P=P\mathcal{B}_\pi$.
\end{proof}
\begin{theorem}
\label{tldthm3} 
Let $\lambda$ be an eigenvalue of $\mathcal{B}_\pi$. Then $\lambda$ is  also an eigenvalue of $\mathcal{A}$.
\end{theorem}
\begin{proof}
Let $x$ be a corresponding eigenvector of $\lambda$. Then $\mathcal{B}_\pi x=\lambda x$. Now, $\mathcal{A} Px=P\mathcal{B}_\pi x=\lambda Px$. Therefore $\lambda$ is an eigenvalue of $\mathcal{A}$.
\end{proof}

We have already seen that $\pi=\lbrace V_{s_1},V_{t_1},\ldots,V_{t_k}\rbrace$ is an equitable partition of $\Gamma$. We rename $\pi$ as $\pi=\lbrace C_1,C_2,\ldots,C_{2k}\rbrace$ where $C_i=V_{s_j}$, if $i=2j-1$, and  $C_i=V_{t_j}$, if $i=2j$. Let $B_{\pi}$ denote the partition  matrix for the partition $\pi$. The matrix $B_\pi=(b_{ij})$ is the square matrix of order $2k$ with $(i,j)$th entry as
$$b_{ij}=\begin{cases}
\vert C_j\vert,&\text{if }i<j,i\text{ odd, } j\text{ even or both }i,j\text{ even},\\
\vert C_i\vert-1,&\text{if }i=j,\ i\text{ even},\\
0,&\text{otherwise.}\\
\end{cases}$$
Thus, if $b=0^{s_1}1^{t_1}\ldots0^{s_k}1^{t_k}$ is the binary string of a threshold graph then
$$B_\pi=\left[\begin{array}{ccccc}
0&t_1&0&\cdots&t_k\\
s_1&t_1-1&0&\cdots&t_k\\
0&0&0&\cdots&t_k\\
\cdots&&\cdots&&\cdots\\
s_1&t_1&s_2&\cdots&t_k-1
\end{array}\right]$$
and $D_\pi=diag(t,t+S_1-1,t-t_1,t+S_2-1,\ldots,n-1)$.\\

Further, let $C_\pi$ be the diagonal matrix with $i$th diagonal equal to $|C_i|$. Since each vertex in a cell has {the same degree $d(C_i)$}, we consider the constants $r_1,r_2,\ldots,r_{2k}$, where $r_i=\frac{|C_i|}{d(C_i)}$, to construct the matrix $X=diag(r_1,r_2,\ldots,r_{2k})$. 

The matrix $\mathcal{B}_\pi$ can be written as $\mathcal{B}_\pi=D_\pi^{-1} A_\pi C_\pi$, where 
\begin{eqnarray*}
A_\pi&=&\left[\begin{array}{ccccc}
0&1&0&\cdots&1\\
1&\beta_1&0&\cdots&1\\
0&0&0&\cdots&1\\
\cdots&&\cdots&&\cdots\\
1&1&1&\cdots&\beta_k
\end{array}\right]\\
&=&A+diag(0,\beta_1,\ldots,0,\beta_k),
\end{eqnarray*}
where $A$ is the adjacency matrix for the threshold graph with binary string $\underbrace{01\ldots 01}_{2k}$ and $\beta_i=1-\frac{1}{|C_{2i}|}=1-\frac{1}{t_i}$.\\
\begin{theorem}
\label{tldthm5}
Let $\Gamma$ be a threshold graph with the binary string $0^{s_1}1^{t_1}\ldots0^{s_k}1^{t_k}$. If $\pi=\{C_1,C_2,\ldots,C_{2k}\}$ is the equitable partition of $\Gamma$, then the eigenvalues of $\mathcal{B}_\pi$ are simple.
\end{theorem}
\begin{proof}
Let $\lambda$ be an eigenvalue of $\mathcal{B}_\pi$ such that the multiplicity of $\lambda$ is $\delta>1$.
Since the matrix $\mathcal{B}_\pi$ is similar to the symmetric matrix $X^{\frac{1}{2}}A_\pi X^{\frac{1}{2}}$, $\mathcal{B}_\pi$ is diagonalizable. Thus, there exists a linearly independent set of $\delta$ eigenvectors 
corresponding to $\lambda$.
Let  $x=[x_1\ x_2\ \ldots\ x_{2k}]^T$ be an eigenvector corresponding to $\lambda$ such that $x_l\neq0$ and $x_m=0\ \forall m>l$, {where $l$ is maximal}. Then $l=2p$, $1\leq p\leq k$. The vector $x$ satisfies the equation
\begin{equation}
\label{eq1}
B_\pi x=\lambda D_\pi x.
\end{equation}
This implies that
\begin{equation}
\label{eq2}
t_p x_{2p}=\lambda d(C_{2p-1})x_{2p-1}.
\end{equation}
Now, using (\ref{eq2}) {along with the $(2p-2)$th and  the $2p$th} equations of (\ref{eq1}), we get
\begin{equation}
[\lambda d(C_{2p})+\lambda s_p+1]x_{2p}=[\lambda d(C_{2p-2})+1]x_{2p-2}.
\end{equation}
Now, iteratively, we get the constants $c_1,c_2,\dots,c_{2p-1}$, to construct $x$ as
$$x=x_{2p}[c_1\ c_2\ \ldots\ c_{2p-1}\ 1\ 0\ \ldots\ 0]^T.$$
Let $x'=[x_1'\ x_2'\ \ldots\ x_{2q}'\ 0\ \ldots\ 0]^T\neq x$ be a vector which satisfies (\ref{eq1}). Then $x+x'$ satisfies equation (\ref{eq1}). If $p=q$ then, by the above arguments we have that $x'$ is a constant multiple of $x$. Again, if $q<p$ then we have $x'=0$.
Then the geometric multiplicity of the eigenvalue $\lambda$ is 1. Hence the proof follows.
\end{proof}
\begin{theorem}
\label{tldthm6}
Let $0^{s_1}1^{t_1}\ldots0^{s_k}1^{t_k}$ be the binary string of a threshold graph $\Gamma$. Then
$$det(B_\pi)=(-1)^ks_1t_1s_2t_2\ldots s_kt_k,$$ and hence $$det(\mathcal{B}_\pi)=(-1)^kr_1r_2\ldots r_{2k}.$$
\end{theorem}
\begin{proof}
We have,
 $B_\pi=A_\pi C_\pi$ and $\mathcal{B}_\pi=D_\pi^{-1} A_\pi C_\pi\sim XA_\pi$, where
 \begin{eqnarray*}
A_\pi&=&\left[\begin{array}{ccccc}
0&1&0&\cdots&1\\
1&\beta_1&0&\cdots&1\\
0&0&0&\cdots&1\\
\cdots&&\cdots&&\cdots\\
1&1&1&\cdots&\beta_k
\end{array}\right],
\end{eqnarray*}
with $\beta_i=1-\frac{1}{|C_{2i}|}=1-\frac{1}{t_i}$.\\
Now we perform some step by step row and column operations to reduce $A_\pi$ into block diagonal form
$$diag(B_1,B_2,\ldots,B_k),$$
 where $B_{i}=\left[\begin{array}{cc}
0&1\\
1&\beta_i
\end{array}\right].$
Hence, $det(A_\pi)=det(B_1)det(B_2)\ldots det(B_k)=(-1)^k$.
Therefore, $$det(B_\pi)=(-1)^k|C_1||C_2|\ldots |C_{2k}|=(-1)^ks_1t_1s_2t_2\ldots s_kt_k, \  \text{ and } \ det(\mathcal{B}_\pi)=(-1)^kr_1r_2\ldots r_{2k}.$$
Hence the result follows.
\end{proof}
\begin{theorem}
\label{tldthm7}
Let $\Gamma$ be a threshold graph. If $\sigma(\mathcal{A})$ and $\sigma(\mathcal{B}_\pi)$ are the spectrum of $\mathcal{A}$ and $\mathcal{B}_\pi$ respectively, then
$$\sigma(\mathcal{A})=\sigma(\mathcal{B}_\pi)\cup T.$$
\end{theorem}
\begin{proof}
We have, $\sigma(\mathcal{A})\supseteq\sigma(\mathcal{B}_\pi)\cup T.$ 
Let $F$ be the set of eigenvectors of the form $f_{t_i}^{(j)},\ 2\leq j\leq t_i, t_i>1$ or $f_{s_i}^{(j)},\ 2\leq j\leq s_i, s_i>1$. Suppose $\lambda$,  which is an eigenvalue of $\mathcal{A}$, is also an eigenvalue of $\mathcal{B}_\pi$. {Let $g$ be an eigenvector corresponding to the eigenvalue $\lambda$ of $\mathcal{B}_\pi$. Then $Pg$ is an eigenvector corresponding to $\lambda$ for the matrix $\mathcal{A}$ and it takes constant value on each cell of $\pi$. Thus, 
$$\langle f,Pg\rangle=0,\ \text{for any } f\in F.$$}
Hence $\sigma(\mathcal{A})=\sigma(\mathcal{B}_\pi)\cup T.$
\end{proof}
\begin{lemma}\cite{Cve1}
\label{tldlem2.3}
Let A be a real matrix such that its row sums are constant (i.e. $Ae=\rho e$ for some $\rho\in \mathbb{R}$). Then all eigenvalues of A, different from $\rho$, are also eigenvalues of any matrix of the form
$$B=A-(ee^T)diag(c_1,c_2\ldots,c_n),$$
where $c_1,c_2,\ldots,c_n\in\mathbb{R}.$
\end{lemma}

\begin{theorem}
\label{tldthm9}
Let $0^{s_1}1^{t_1}\ldots0^{s_k}1^{t_k}$ be the binary string of a threshold graph $\Gamma$. If $\lambda_1$ is the smallest eigenvalue of the normalized adjacency matrix of $\Gamma$, then  $$-\frac{n-t_k}{n-1}\leq\lambda_1\leq-\frac{1}{t+s_1-1}.$$
\end{theorem}
\begin{proof}
It is sufficient to prove the theorem for the matrix $\mathcal{B}_\pi$. Consider $B=\mathcal{B}_\pi-ef^T$, where $e$ is the column vector with all entries equal to 1 and $f$ is any real column vector. Since $\mathcal{B}_\pi$ is stochastic, by Lemma \ref{tldlem2.3}, any eigenvalue $\lambda$ of $\mathcal{B}_\pi$, other than 1, is also an eigenvalue of $B$.  
Now, we choose $f=[f_i]$ where  
$$f_i=\begin{cases}\frac{t_k-1}{n-1},&\text{if } i=k,\\
0,&\text{otherwise.}
\end{cases}$$
By Ger{\v s}gorin disk theorem, we have
\begin{eqnarray*}
|\lambda_1|&\leq& 1-\frac{t_k-1}{n-1}\\
&=&\frac{n-t_k}{n-1}.
\end{eqnarray*}
Therefore, the smallest eigenvalue of $\mathcal{A}(\Gamma)$ is bounded below by $-\frac{n-t_k}{n-1}$.\\

Now we consider the $2\times2$ principle submatrix $\mathcal{B}^\ast$ of  $\mathcal{A}$ by taking $i$th and $j$th row such that $i\in C_1$ and $j\in C_2$. Then $$\mathcal{B}^\ast=\left[\begin{array}{cc}
0&\frac{1}{t}\\
\frac{1}{t+s_1-1}&0
\end{array}\right].$$ 
By Cauchy interlacing theorem we have
$$\lambda_1\leq -\frac{1}{t+s_1-1}.$$
Hence the theorem follows.
\end{proof}
\begin{example}{Let $\Gamma$ be a threshold graph with the binary string $b=0011100011$.  We have
$$\mathcal{B}_\pi=\left[\begin{array}{ccccccccc}
0&\frac{3}{5}&0&\frac{2}{5}\\
\frac{1}{3}&\frac{1}{3}&0&\frac{1}{3}\\
0&0&0&1\\
\frac{2}{9}&\frac{3}{9}&\frac{3}{9}&\frac{1}{9}
\end{array}\right].$$

The eigenvalues of $\mathcal{B}_\pi$ are $-0.6063$, $-0.3072$, $0.3579$, $1$.

Again, by Lemma \ref{tldlem1}, $0$, $-0.2$ and $-0.1111$ are the eigenvalues of $\mathcal{A}$ with the multiplicitties 3, 2 and 1, respectively.

Thus, the complete spectrum of $\mathcal{A}$ is $-0.6063$, $-0.3072$, $-0.2^2$, $-0.1111$, $0^3$, $0.3579$, $1$.}
\end{example}
\section{The normalized adjacency energy of threshold graphs}

The \textit{normalized adjacency energy} of a graph $\Gamma$ is defined by 
$$E_\mathcal{A}(\Gamma)=\sum |\lambda|.$$
It has been seen that matrix energy of a graph has importance in spectral graph theory and chemical graph theory. The normalized adjacency energy is equal to its normalized Laplacian energy $E_\mathcal{L} (\Gamma)$. Researchers found bounds for normalized adjacency energy in terms of general Randi\'c index \cite{Mcav} of that graph. The general Randi\'c index $R_{-1}(\Gamma)$ is defined by 
$$R_{-1}(\Gamma)=\sum_{i\sim j}\frac{1}{d_id_j}.$$ The general Randi\'c index $R_{-1}(\Gamma)$ for a threshold graph $\Gamma$ can be calculated explicitly.
\begin{theorem}
Let $0^{s_1}1^{t_1}\cdots 0^{s_k}1^{t_k}$ be the binary string of a threshold graph $\Gamma$. Then
$$R_{-1}(\Gamma)=\sum_{i=1}^{k-1}\Big{[}(r_{2i}+r_{2i+1})\sum_{j=i+1}^kr_{2j}\Big{]}+ \sum_{i=1}^k\frac{\alpha_i}{(t+S_i-1)^2},$$
where $\alpha_i=\begin{cases}
{t_i \choose 2},& \text{if }t_i>1,\\
0,& \text{if }t_i=1.
\end{cases}$
\end{theorem} 
\begin{proof}
We have 
\begin{eqnarray*}
R_{-1}(\Gamma)&=&\sum_{i\sim j}\frac{1}{d_id_j}\\
&=&\sum_{\substack{i\sim j\\ i,j\in C_l}}\frac{1}{d_id_j}+\sum_{\substack{i\sim j\\ i\in C_m,j\in C_l\\ m\neq l}}\frac{1}{d_id_j}\\\
&=&R_{-1}'(\Gamma)+R_{-1}''(\Gamma).
\end{eqnarray*}
Since, for any edge $e=(ij)$, when the endvertices  $i$ and $j$ {are in the same cell}  $C_l$, $l$ is even and $|c_l|>1$. So, 
$$R_{-1}'(\Gamma)=\sum_{i=1}^k\frac{\alpha_i}{(t+S_i-1)^2}.$$ 
Now,
\begin{equation}
\label{tldeqn4}
R_{-1}''(\Gamma)=\sum_{\substack{i\sim j\\ i\in C_1,j\in C_m\\ m>1}}\frac{1}{d_id_j}
+\sum_{\substack{i\sim j\\ i\in C_2,j\in C_m\\ m>2}}\frac{1}{d_id_j}+\cdots+\sum_{\substack{i\sim j\\ i\in C_{2k-1},j\in C_{2k}}}\frac{1}{d_id_j}.
\end{equation}
The first term of this equation is equal to $\frac{s_1}{t}\sum_{i=1}^k \frac{t_i}{t+S_i-1}$, the second term is equal to $\frac{t_1}{t+S_1-1}\sum_{i=2}^k \frac{t_i}{t+S_i-1}$, and so on. Therefore,
\begin{eqnarray*}
R_{-1}''(\Gamma)&=&r_1\sum_{i=1}^kr_{2i}+r_2\sum_{i=2}^kr_{2i} +\cdots+r_{2k-1}r_{2i}\\
&=&\sum_{i=1}^{k-1}\Big{[}(r_{2i}+r_{2i+1})\sum_{j=i+1}^kr_{2j}\Big{]}.
\end{eqnarray*}
Hence the result follows.
\end{proof}
\begin{theorem}
\label{tldthm10}
Let $0^{s_1}1^{t_1}\cdots 0^{s_k}1^{t_k}$ be the binary string of a threshold graph $\Gamma$. Then
$$2\bigg[\frac{k(n-t_k)}{n-1}+\sum_{i=1}^k\frac{t_i-1}{t+S_i-1}\bigg{]}\geq E_\mathcal{A}(\Gamma)\geq 2\bigg{[}\frac{k}{t}+\sum_{i=1}^k\frac{t_i-1}{t+S_i-1}\bigg{]}.$$
\end{theorem}
\begin{proof}
Let $\mu_1<\mu_2<\ldots<\mu_{2k}=1$ be the eigenvalues of $\mathcal{B}_\pi$. Then we have 
\begin{eqnarray*}
\sum_{i=1}^{2k}|\mu_i|&=& 2\big{|}\sum_{i=1}^{k}\mu_i\big{|}+tr(\mathcal{B}_\pi)\\
&\geq& 2k|\mu_k|+\sum_{i=1}^{k}\frac{|C_{2i}|-1}{d(C_{2i})}\\
&\geq&\frac{2k}{t}+\sum_{i=1}^{k}\frac{t_i-1}{t+S_i-1}.
\end{eqnarray*}
Using Lemma \ref{tldlem1} and Theorem \ref{tldthm7},  we get
\begin{eqnarray*}
E_\mathcal{A}(\Gamma)&\geq& \frac{2k}{t}+\sum_{i=1}^k\frac{t_i-1}{t+S_i-1}+\sum_{i=1}^k\frac{t_i-1}{t+S_i-1}\\
&\geq& 2\Bigg{[}\frac{k}{t}+\sum_{i=1}^k\frac{t_i-1}{t+S_i-1}\Bigg{]}.
\end{eqnarray*}
Again, by Theorem \ref{tldthm9}, we have
\begin{eqnarray*}
\sum_{i=1}^{2k}|\mu_i|&=& 2\sum_{i=1}^{k}|\mu_i|+tr(\mathcal{B}_\pi)\\
&\leq& 2k|\mu_1|+\sum_{i=1}^{k}\frac{|C_{2i}|-1}{d(C_{2i})}\\
&\leq&\frac{2k(n-t_k)}{n-1}+ \sum_{i=1}^{k}\frac{t_i-1}{t+S_i-1}.
\end{eqnarray*}
Therefore,
$$E_\mathcal{A}(\Gamma)\leq2\Bigg{[}\frac{k(n-t_k)}{n-1}+\sum_{i=1}^k\frac{t_i-1}{t+S_i-1}\Bigg{]}.$$
Hence the proof follows.
\end{proof}


\section{Threshold graphs with {a} small number of partitions}
\label{section3}
In this section we discuss spectral properties of threshold graphs where $k=1$ or $k=2$. Later, in this section we characterize threshold graphs with a few distinct normalized eigenvalues.\\

\textbf{ Case I: $k=1$}\\
If $k=1$ then the binary string of $\Gamma$ is $0^s1^t$. Now if $s=1$ then $\Gamma$ is the complete graph $K_n$, whereas, if $t=1$  then $\Gamma$ is the star $S_n$. So in these two cases {the normalized  eigenvalues} of $\Gamma$ are $1,(\frac{-1}{n-1})^{n-1}$ and $1,0^{n-2} ,-1$ respectively. Let $s,t>1$. Then 
$$\mathcal{B}_\pi=\left[\
\begin{array}{cc}
0&1\\
\frac{s}{n-1}&\frac{t-1}{n-1}
\end{array}\right].$$
Therefore, the eigenvalues of {$\mathcal{A}$} are $1$, $0^{s-1}$, $(\frac{-1}{n-1})^{t-1}$ and $-\frac{n-t}{n-1}$.\\

\textbf{ Case II: $k=2$}\\
If $k=2$ then the binary string of $\Gamma$ is of the form $0^{s_1}1^{t_1}0^{s_2}1^{t_2}$, then 
$$\mathcal{B}_\pi=\left[\
\begin{array}{cccc}
0&\frac{t_1}{t}&0&\frac{t_2}{t}\\
\frac{s_1}{t+s_1-1}&\frac{t_1-1}{t+s_1-1}&0&\frac{t_2}{t+s_1-1}\\
0&0&0&1\\
\frac{s_1}{n-1}&\frac{t_1}{n-1}&\frac{s_2}{n-1}&\frac{t_2-1}{n-1}
\end{array}\right].$$
Now, we have the following theorem.
\begin{theorem}
\label{tldthm4.1}
Let $\Gamma$ be a threshold graph with the binary string $b=0^{s_1}1^{t_1}0^{s_2}1^{t_2}$. Then\\ 
(a) the multiplicity of the {normalized  eigenvalue} $\frac{-1}{n-1}$ is $t_2-1$,\\
(b) the multiplicity of the {normalized  eigenvalue} $\frac{-1}{t}$ is $t_1$ if and only if $s_1=1$ or $s_2(tt_1-1)+1$. 
\end{theorem}
\begin{proof}
(a) {By Lemma \ref{tldlem1}, $\frac{-1}{n-1}$ is an eigenvalue of $\mathcal{A}$ with multiplicity at least $t_2-1$. Thus we only have to show that $\frac{-1}{n-1}$ is not an eigenvalue of $\mathcal{B}_\pi$. Suppose, for contradiction, that $\frac{-1}{n-1}$ is an eigenvalue of $\mathcal{B}_\pi$} with corresponding eigenvector $x=[x_1\ x_2\ x_3\ x_4]^T$. The eigenvalue equations {of $\mathcal{B}_\pi$} for the eigenvalue $\frac{-1}{n-1}$ are
\begin{equation}
\label{tldeqn5}
\left.
\begin{aligned}
t_1x_2+t_2x_4&=-\frac{t}{n-1}x_1,\\
s_1x_1+t_1x_2+t_2x_4&=\frac{s_2}{n-1}x_2,\\
x_4&=-\frac{1}{n-1}x_3,\\
s_1x_1+t_1x_2+s_2x_3+t_2x_4&=0.
\end{aligned}
\right\}
\end{equation}
Solving (\ref{tldeqn5}), we have
$$x_1=x_2=x_3=x_4=0.$$
Hence the multiplicity of the {normalized  eigenvalue} $-\frac{1}{n-1}$ is exactly $t_2-1$.\\

(b) Let the multiplicity of the {normalized  eigenvalue} $-\frac{1}{t}$ is $t_1$. Then $-\frac{1}{t}$ is also an eigenvalue of $\mathcal{B}_\pi$ with corresponding eigenvector $x=[x_1\ x_2\ x_3\ x_4]^T$. Then from the eigenvalue equations {of $\mathcal{B}_\pi$} for the eigenvalue $\frac{-1}{t}$, we have
\begin{eqnarray*}
x_1+t_1x_2+t_2x_4&=&0,\\
s_1x_1+(t_1+\frac{s_1-1}{t})x_2+t_2x_4&=&0,\\
\frac{1}{t}x_3+x_4&=&0,\\
s_1x_1+t_1x_2+s_2x_3+(t_2+\frac{s-1}{t})x_4&=&0.
\end{eqnarray*}
For a non-trivial solution, we must have
$$\left|\
\begin{array}{cccc}
1&t_1&0&t_2\\
s_1&t_1+\frac{s_1-1}{t}&0&t_2\\
0&0&\frac{1}{t}&1\\
s_1&t_1&s_2&t_2+\frac{s-1}{t}=0
\end{array}\right|=0,$$
that is, $$\frac{1}{t^3}(s_1-1)(1-t^2)(s-1-tt_1s_2)=0.$$
Since $t>1$, we have, either
$s_1=1\text{ or }s_1=s_2(t_1t-1)+1.$ Hence the result follows.
\end{proof}
Now, consider the matrix 
$$M=\left[\
\begin{array}{cccc}
0&1&0&0\\
1&1&0&0\\
0&1&1&0\\
0&1&0&1
\end{array}\right],$$
which is a non-singular matrix. Now
$$M^{-1}\mathcal{B}_\pi M=\left[\
\begin{array}{cccc}
1&\frac{s_1}{t+s_1-1}&0&\frac{t_2}{t+s_1-1}\\
0&-\frac{s_1}{t+s_1-1}&0&\frac{t_2(s_1-1)}{t(t+s_1-1)}\\
0&-\frac{s_1}{t+s1-1}&0&\frac{t_1+s_1-1}{t+s_1-1}\\
0&-\frac{s_1s_2}{(t+s_1-1)(n-1)}&\frac{s_2}{n-1}&-\frac{s_2t_2+t+s_1-1}{(t+s_1-1)(n-1)}
\end{array}\right]=\left[\begin{array}{cc}
1&\textbf{x}\\
O^T&B^*
\end{array}\right]$$
where $B^*=\frac{1}{t+s_1-1}\left[\begin{array}{ccc}
-s_1&0&\frac{t_2(s_1-1)}{t}\\-s_1&0&t_1+s_1-1\\-\frac{s_1s_2}{n-1}&\frac{s_2(t+s_1-1)}{n-1}&-\frac{s_2t_2+t+s_1-1}{n-1}
\end{array}\right].$

In particular, if $s_1=1$, then the eigenvalues of {$\mathcal{A}$} are 
$1,0^{s-2},$ $(\frac{-1}{t})^{t_1},$ $(\frac{-1}{n-1})^{t_2-1},$ $\frac{1}{2t(n-1)}\Big{[}-(s_2t_2+t)\pm\sqrt{(s_2t_2+t)^2+4t_1s_2t(n-1)}\Big{]}.$\\

A \textit{pineapple graph} (see Figure \ref{tldfig4}(a)) is obtained by appending pendant vertices to a vertex of a complete graph (of at least three vertices). The pineapple graph is a threshold graph with exactly one dominating vertex. The general form of the binary string of a pineapple graph is $01^{t-1}0^{n-t-1}1$. Using the above argument, we get the {normalized eigenvalues} of a pineapple which are
$$1,\ 0^{n-t-2},\ (\frac{-1}{t})^{t-1},\ \frac{1}{2t}\Bigg{[}-1\pm \sqrt{1+\frac{4t(t-1)(n-t-1)}{n-1}}\Bigg{]}.$$

\subsection{Threshold graph with at most five distinct eigenvalues}

\begin{theorem}
Let $\Gamma$ be a threshold graph on n vertices. Then $\Gamma$ has\\
(a) two distinct {normalized eigenvalues} if and only if $b=01^{n-1}$,\\
(b) three distinct {normalized eigenvalues} if and only if $b=0^{n-1}1$,\\
(c) four distinct {normalized eigenvalues} if and only if $b=0^s1^{n-s},\ 1<s<n-1$ or $b=01^{n-3}01$.
\end{theorem}
\begin{proof}
Since all the eigenvalues of $\mathcal{B}_\pi$ are distinct, $\Gamma$ has at most four distinct {normalized eigenvalues}, and we must have $k=1$ or $k=2$. The proofs of (a) and (b) are straight forward, since the complete graph is the only graph with two distinct {normalized eigenvalues}, whereas, star is the only threshold graph with three distinct {normalized eigenvalues}.\\
 
 (c) Let $\Gamma$ have four distinct {normalized eigenvalues}, namely, $\lambda_1< \lambda_2<\lambda_3<\lambda_4=1$. Here, we have two cases.\\
 \textbf{case I. }If $\lambda_3=0$, then k=1. Let $b=0^s1^{n-s}$. Note that the value of $s$ cannot be equal to 1 or $n-1$. If $1<s<n-1$, then the eigenvalues of {$\mathcal{A}$} are $1$, $0^{s-1}$, $(\frac{-1}{n-1})^{n-s-1}$ and $-\frac{s}{n-1}$. Therefore, $b=0^s1^{n-s},\ 1<s<n-1$.\\
\textbf{Case II. }If $\lambda_3>0$, then $k=2$. Let $b=0^{s_1}1^{t_1}0^{s_2}1^{t_2}$. Since $0$ is not an eigenvalue of {$\mathcal{A}$}, therefore, $s_1=1$ and $s_2=1$. Since $s_1=1$, $-\frac{1}{t}$ is an eigenvalue of $\mathcal{B}_\pi$. Again, since $-\frac{1}{n-1}$ is not an eigenvalue of $\mathcal{B}_\pi$, we have $t_2=1$, otherwise $\Gamma$ must have five distinct { normalized eigenvalues}. Therefore $b=01^{n-3}01.$
Thus the proof is obtained by combining Case I and Case II.
\end{proof}
Next, we characterize threshold graphs with five distinct  {normalized eigenvalues}. If a {threshold graph has} five distinct  {normalized eigenvalues}, then $k$ should be equal to 2. We have already seen that a pineapple graph has exactly five distinct  {normalized eigenvalues}. Now, in the following theorem, we obtain all possible threshold graphs with five distinct {normalized eigenvalues}.

\begin{figure}[h]
\centering
\includegraphics[width=12.5cm]{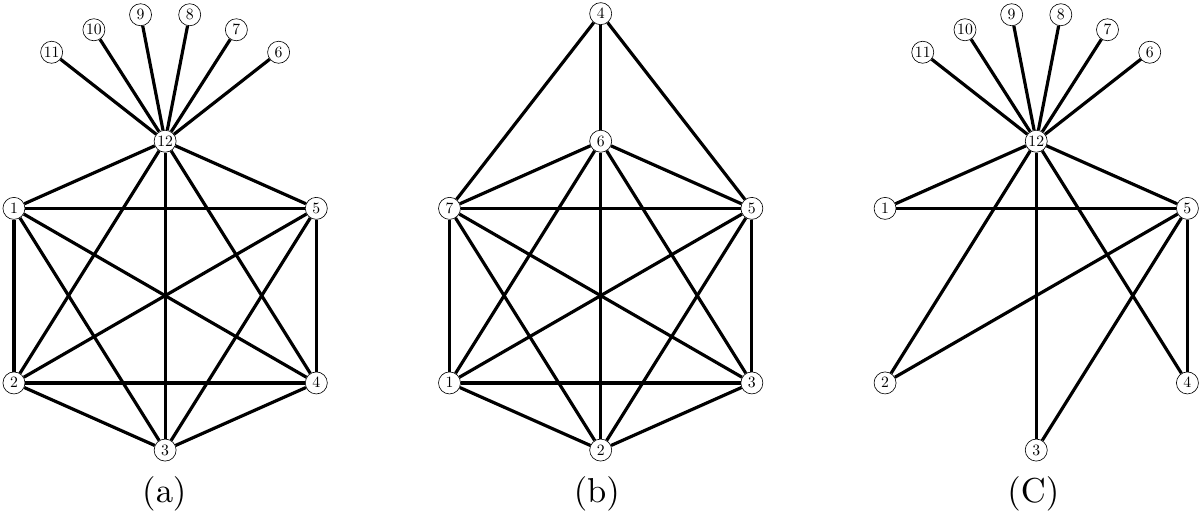}
\caption{Threshold graphs with five distinct eigenvalues.} \label{tldfig4}
\end{figure}
\begin{theorem}
Let $\Gamma$ be a threshold graph on $n\geq5$ vertices. Then $\Gamma$ has five distinct {normalized eigenvalues} if and only if one of the following conditions holds\\
(i) $b=01^{n-t-2}01^t$, $t>1$,\\
(ii) $b=0^s10^{n-s-2}1$,\\
(iii) $b=0^{s_1}1^{t-1}0^{s_2}1$, such that, $s_1=s_2(t^2-t-1)+1,$\\
(iv) $\Gamma$ is a pineapple graph.
\end{theorem}
\begin{proof}
Let $\lambda_1<\lambda_2<\lambda_3<\lambda_4<\lambda_5=1$ are the eigenvalues of {$\mathcal{A}$}. So, $k=2$. Let $b=0^{s_1}1^{t_1}0^{s_2}1^{t_2}$. Here,  two cases are possible:\\ 
\textbf{Case I.} If $\lambda_3<0$. Then $s_1=1$, $s_2=1$ and $-\frac{1}{t}$ is an eigenvalue of $\mathcal{B}_\pi$. We also have $\lambda_3=-\frac{1}{n-1}$. Therefore, by Theorem \ref{tldthm4.1}, we get $t_2>1$. Hence $b=01^{n-t-2}01^t,\ t>1$.\\
 \textbf{Case II.} Let $\lambda_3=0$, then $s>2$ and all four nonzero eigenvalues {of $\mathcal{A}$} are also the eigenvalues of $\mathcal{B}_\pi$. Since $-\frac{1}{n-1}$ is not an eigenvalue of $\mathcal{B}_\pi$, therefore, $t_2=1$. Let $b=0^{s_1}1^{t-1}0^{s_2}1$. Now, we have two subcases.\\
 \textbf{Subcase (a)} Let $t=2$. Then the eigenvalues of $\mathcal{B}_\pi$ are the nonzero eigenvalues of {$\mathcal{A}$}. Eventually, $b=0^s10^{n-s-2}1$. \\
 \textbf{Subcase (b)} If $t>2$, $-\frac{1}{t}$ is an eigenvalue of $\mathcal{A}$. Therefore $\Gamma$ has five distinct {normalized eigenvalues}, only if $-\frac{1}{t}$ is also an eigenvalue of $\mathcal{B}_\pi$. Thus, by Theorem \ref{tldthm4.1}, we have $s_1=1$ or $s_1=s_2(t^2-t-1)+1$. Now, if $s_1=1$ then $b=01^{t-1}0^{n-t-1}1$, and thus $\Gamma$ is a pineapple graph. Otherwise, $b=0^{s_1}1^{t-1}0^{s_2}1$ with $s_1=s_2(t^2-t-1)+1$. Hence, the proof is completed.
\end{proof}
\subsection*{Conclusions}
In this article we have studied the {normalized eigenvalues} of threshold graphs. We have also characterized threshold graphs, with the binary string $0^{s_1}1^{t_1}\ldots0^{s_k}1^{t_k}$, for small values of $k$. Now our observation, from  
 Theorem \ref{tldthm1} and Theorem \ref{tldthm5}, is that a threshold graph with the binary string $0^{s_1}1^{t_1}\ldots0^{s_k}1^{t_k}$ has exactly $k$ distinct positive normalized eigenvalues. This implies that the binary strings for
 two cospectral threshold graphs must have the same value for $k$. Thus, the question that arises is: are two threshold graphs having the same {normalized eigenvalues}  always isomorphic?   

\section*{Acknowledgements}
We are very grateful to the referees for their comments and suggestions, which helped to improve the manuscript. We also thankful to Asok K.~Nanda for his kind suggestions during the writing of the manuscript.  Ranjit Mehatari is supported by CSIR Grant No. 09/921(0080)/2013-EMR-I, India.

\end{document}